\documentclass{amsart}

\usepackage[utf8]{inputenc}
\usepackage[T1]{fontenc}
\usepackage{amsmath,amssymb,amsthm,mathtools}
\usepackage{enumitem}
\usepackage{comment}
\usepackage{mathrsfs}
\usepackage[colorlinks=true,linkcolor=blue,citecolor=blue,urlcolor=blue]{hyperref}

\newtheorem{theorem}{Theorem}[section]
\newtheorem{proposition}[theorem]{Proposition}
\newtheorem{lemma}[theorem]{Lemma}
\newtheorem{corollary}[theorem]{Corollary}
\theoremstyle{remark}
\newtheorem{remark}[theorem]{Remark}

\newcommand{\N}{\mathbb N}
\newcommand{\C}{\mathbb C}
\newcommand{\T}{\mathbb T}

\newcommand{\supp}{\mathrm{supp}}
\renewcommand{\span}{\mathrm{span}}
\newcommand{\dist}{\mathrm{dist}}

\title[$B(H)$ is not a twisted groupoid $C^*$-algebra]{$B(H)$ is not a twisted groupoid $C^*$-algebra}

\author[Alcides Buss]{Alcides Buss}
\address[Alcides Buss]{Departamento de Matem\'atica, Universidade Federal de Santa Catarina, 88.040-900
Florian\'opolis-SC, Brazil}
	\email{alcides.buss@ufsc.br}
	\urladdr{http://mtm.ufsc.br/~alcides/}

\author[Luiz F. Garcia]{Luiz Felipe Garcia}
\address[Luiz F. Garcia]{Departamento de Matem\'atica, Universidade Federal de Santa Catarina, 88.040-900
Florian\'opolis-SC, Brazil}
	\email{lfgarcia98@gmail.com}

\author[Tomás Pacheco]{Tomás Pacheco}
\address[Tomás Pacheco]{Center for Mathematical Analysis, Geometry and Dynamical Systems,
Department of Mathematics, Instituto Superior T\'ecnico, University of Lisbon,
Av. Rovisco Pais 1, 1049-001 Lisboa, Portugal}
	\email{tomas.pacheco@tecnico.ulisboa.pt}
\begin{document}
\maketitle

\begin{abstract}
We show that $B(H)$ for an infinite dimensional Hilbert space $H$ cannot be realized as the reduced twisted $C^*$-algebra of any locally compact Hausdorff étale groupoid.

The proof is based on the canonical conditional expectation $$C_r^*(G,\Sigma)\to C_0(G^{(0)})$$ and a structural analysis of the resulting diagonal subalgebra inside $B(H)$. We show that this diagonal must be an atomic abelian von Neumann algebra, and then exclude both possibilities for its spectrum.

If the unit space is finite, one obtains a tracial state on $C_r^*(G,\Sigma)$, which is impossible for $B(H)$. If it is infinite, the groupoid structure forces a block-sparsity phenomenon for compactly supported sections, which is incompatible with $B(H)$.

This provides the first examples of $C^*$-algebras that cannot be realized as reduced twisted étale groupoid $C^*$-algebras.

\end{abstract}

\section{Introduction}

Étale groupoids and their $C^*$-algebras provide a powerful framework for
encoding dynamical systems, inverse semigroup actions, Cartan pairs, and many
other constructions in operator algebras; see for instance
\cite{Renault1980,Kumjian1986,Renault2008}.

A fundamental problem is to understand which $C^*$-algebras arise as
groupoid $C^*$-algebras. In \cite{BussSims} it was shown that every groupoid $C^*$-algebra is isomorphic to its opposite algebra, yielding obstructions in the untwisted setting. However, this argument does not extend to twisted groupoid $C^*$-algebras, which need not be self-opposite. It remained open whether every $C^*$-algebra is isomorphic to a twisted groupoid $C^*$-algebra.
This question appears implicitly in \cite{BussSims} and is explicitly raised in the literature, e.g. in \cite{ClarkOceallaighPham2025}, as well as in recent MathOverflow discussions \cite{garciaMO, pkoaMO}.
To the best of our knowledge, no example of a $C^*$-algebra failing to admit such a realization was previously known. 

The goal of this paper is to provide the first such examples.

\begin{theorem}\label{thm:main}
Let $H$ be an infinite-dimensional Hilbert space.
Then there is no locally compact Hausdorff étale groupoid $G$ and no twist
$\Sigma$ over $G$ such that
\[
B(H)\cong C_r^*(G,\Sigma).
\]
\end{theorem}

In contrast, if $\dim(H)=n<\infty$, then $B(H)$ is isomorphic to the groupoid
$C^*$-algebra of the finite pair groupoid on $\{1,\dots,n\}$.
Moreover, if $H=\ell^2(X)$ for an arbitrary set $X$, then $B(H)$ arises as the
von Neumann algebra of the measured pair groupoid $X\times X$ equipped with
counting measure. Thus the obstruction we obtain is genuinely $C^*$-algebraic
and topological, and does not appear at the von Neumann level.

It is also worth noting that $B(H)$ is always self-opposite. Therefore, our
result provides the first examples of $C^*$-algebras that are self-opposite but
cannot be realized as (twisted) étale groupoid $C^*$-algebras.

The proof is based on the canonical diagonal subalgebra
\(
A=C_0(G^{(0)})\subseteq C_r^*(G,\Sigma)
\)
and the faithful conditional expectation
\(
E:C_r^*(G,\Sigma)\to C_0(G^{(0)}).
\)
Transporting this structure through an isomorphism
\[
C_r^*(G,\Sigma)\cong B(H),
\]
we are led to study commutative $C^*$-subalgebras
\(
A\subseteq B(H)
\)
admitting a faithful conditional expectation.

The argument proceeds in three steps:

\begin{enumerate}[label=\rm(\arabic*)]
\item show that $A$ is a von Neumann algebra;
\item show that $A$ is atomic;
\item exclude the finite and infinite atomic cases.
\end{enumerate}

The first step is a general operator-algebraic consequence of the existence of a
faithful conditional expectation. The second step provides a strong structural
restriction on the diagonal. The final step uses the groupoid origin of $A$:

\begin{itemize}
\item if $G^{(0)}$ is finite, then $C_r^*(G,\Sigma)$ admits a tracial state,
      contradicting $B(H)$;
\item if $G^{(0)}$ is infinite and discrete, then compactly supported sections
      give rise to operators with uniformly bounded propagation between the
      summands of $H$, yielding a block-sparsity property incompatible with
      $B(H)$.
\end{itemize}

We emphasize that the examples obtained here are necessarily non-separable as
$C^*$-algebras (even when $H$ is separable). Nevertheless, they provide the first
evidence that large classes of $C^*$-algebras -- particularly infinite von Neumann
algebras -- may fail to admit realizations as twisted étale groupoid $C^*$-algebras.

After the first version of this paper was announced, we were informed that David Gao had independently been working on this problem, following a discussion on MathOverflow \cite{pkoaMO}. In a private communication, he shared a preliminary, unpublished manuscript containing a similar, but different proof of our main theorem. His approach to the atomicity (our Theorem~\ref{thm:atomic}) of the diagonal relies on the normal-singular decomposition of the conditional expectation \cite[Sections~III.2 and III.3]{TakesakiI}. In contrast, our approach is entirely self-contained and explicitly uses the block-sparsity of compactly supported sections (see Section~\ref{sec:inf-atomic-case}).

\subsection*{Acknowledgements}
The first author was supported by CNPq and Fapesc. The second author was supported by CAPES. The third author was funded by FCT/Portugal and the Recovery and Resilience Plan (PRR) through projects UID/04459/2025 and UID/PRR/04459/2025.

The authors are grateful to Ruy Exel for helpful and fruitful discussions.

\section{Faithful expectations and von Neumann subalgebra}

We begin with a general observation, which we will apply to the case $M=B(H)$.

\begin{proposition}\label{prop:A-vN}
Let $M$ be a von Neumann algebra and let $A\subseteq M$ be a commutative
unital $C^*$-subalgebra. Suppose there exists a faithful conditional
expectation $E:M\to A$. Then $A$ is a von Neumann subalgebra of $M$.
\end{proposition}

\begin{proof}
Let $(a_i)$ be an increasing bounded net of self-adjoint elements of $A$.
Since $M$ is a von Neumann algebra, the supremum $a=\sup_i a_i$ exists in $M$.

For each $i$ we have $a_i\le a$, hence applying $E$ gives
$a_i=E(a_i)\le E(a)$. Taking suprema yields $a\le E(a)$.
Since $E(E(a)-a)=0$ and $E$ is faithful, $E(a)=a\in A$.

Hence $A$ is monotone complete. Moreover, if $\omega$ is any normal state on
$M$, then its restriction to $A$ is normal, since suprema of increasing nets
in $A$ agree with those in $M$. As normal states on $M$ separate points of
$M_+$, their restrictions separate points of $A_+$. By Kadison's theorem
\cite{Kadison1985} (see also \cite[Theorem~III.3.16]{TakesakiI}), it follows
that $A$ is a von Neumann algebra.

Finally, the inclusion $A\hookrightarrow M$ preserves suprema of bounded
increasing nets of self-adjoint elements, hence is normal. Therefore $A$ is a
von Neumann subalgebra of $M$.
\end{proof}

\section{Faithful expectations and atomicity}

We first record a standard fact about diffuse abelian von Neumann algebras.

\begin{proposition}\label{prop:diffuse-partition}
Let $D$ be a diffuse abelian von Neumann algebra, and let $\phi$ be a normal
positive functional on $D$. Then for every $\varepsilon>0$, there is a finite
partition of unity by projections
\[
1=e_1+\cdots+e_n
\]
such that $\phi(e_k)<\varepsilon$ for all $k$.
\end{proposition}

\begin{proof}
By \cite[Theorem~III.1.18]{TakesakiI}, there is a measure space $(X,\mu)$ such
that $D\cong L^\infty(X,\mu)$. Since $D$ is diffuse, $\mu$ is atomless. Under
this identification, $\phi$ is given by integration against some
$g\in L^1(X,\mu)_+$. Thus the finite measure $\nu$ defined by $d\nu=g\,d\mu$
is atomless. By \cite[211Y(c)]{Fremlin2}, there is a finite measurable partition
$X=E_1\sqcup\cdots\sqcup E_n$ such that $\nu(E_k)<\varepsilon$ for all $k$.
The projections $e_k:=\chi_{E_k}$ have the required properties.
\end{proof}

We now prove the crucial structural theorem.

\begin{theorem}\label{thm:atomic}
Let $A\subseteq B(H)$ be a commutative von Neumann algebra and suppose there
exists a faithful conditional expectation $E:B(H)\to A$. Then $A$ is atomic.
\end{theorem}

\begin{proof}
Write $A=A_a\oplus A_d$, where $A_a$ is atomic and $A_d$ is diffuse.
Let $p\in A$ be the central projection with $A_d=Ap$. We show that $p=0$.

Assume $p\neq 0$. Then $A_d$ is a nonzero diffuse commutative von Neumann algebra.
Set $H_p:=pH$. Since $A_d=Ap$, it acts nondegenerately on $H_p$, and we define
$E_p:B(H_p)\to A_d$ by $E_p(x):=E(pxp)$. This is a faithful conditional expectation.

Indeed, positivity and $A_d$-bimodularity are immediate. If $x\in B(H_p)$ is
positive and $E_p(x)=0$, then, viewing $x$ as an operator on $H$ with support in
$p$, we have $x=pxp$ and hence
\[
E(x)=E(pxp)=E_p(x)=0.
\]
Since $E$ is faithful, it follows that $x=0$.

Choose a unit vector $\xi\in H_p$ and consider the normal positive functional
\[
\phi:A_d\to\C,\qquad \phi(a):=\langle a\xi,\xi\rangle.
\]
Fix $\varepsilon>0$. Applying Proposition~\ref{prop:diffuse-partition} to $A_d$,
whose unit is $p$, gives pairwise orthogonal projections $e_1,\dots,e_n\in A_d$
such that
\[
e_1+\cdots+e_n=p
\qquad\text{and}\qquad
\|e_k\xi\|^2=\phi(e_k)<\varepsilon
\quad (k=1,\dots,n).
\]
Let $q:=|\xi\rangle\langle\xi|$ and set
\[
x_\varepsilon:=\sum_{k=1}^n e_k q e_k.
\]
Using the $A_d$-bimodularity of $E_p$ and the commutativity of $A_d$, we obtain
\[
E_p(x_\varepsilon)=\sum_k e_kE_p(q)e_k
=\sum_k e_kE_p(q)=pE_p(q)=E_p(q).
\]
For each $k$ we have
\(
e_kqe_k=|e_k\xi\rangle\langle e_k\xi|,
\)
hence
\(
\|e_kqe_k\|=\|e_k\xi\|^2<\varepsilon.
\)
Since the ranges of the operators $e_kqe_k$ are pairwise orthogonal,
\[
\|x_\varepsilon\|=\max_k \|e_kqe_k\|<\varepsilon.
\]
Therefore
\[
\|E_p(q)\|=\|E_p(x_\varepsilon)\|\le \|x_\varepsilon\|<\varepsilon.
\]
As $\varepsilon>0$ was arbitrary, it follows that $E_p(q)=0$. By faithfulness of
$E_p$, we get $q=0$, a contradiction. Hence $p=0$.
\end{proof}

\begin{corollary}\label{cor:A-ellinftyX}
Let $A\subseteq B(H)$ be a commutative unital $C^*$-subalgebra admitting a
faithful conditional expectation
\(
E:B(H)\to A.
\)
Then
\[
A\cong \ell^\infty(X)
\]
for some index set $X$ with $|X|\le \dim(H)$.
\end{corollary}
\begin{proof}
By Proposition~\ref{prop:A-vN}, $A$ is a commutative von Neumann algebra.
By Theorem~\ref{thm:atomic}, it is atomic.
Hence $A$ is isomorphic to $\ell^\infty(X)$ for some index set $X$ (see \cite[Proposition~III.1.19]{TakesakiI} and the subsequent discussion).

Let $(p_x)_{x\in X}$ be the minimal projections of $A$ corresponding to the
coordinate functions in $\ell^\infty(X)$. Since each $p_x\neq 0$, the subspace
$p_xH$ is nonzero. Moreover, for $x\neq y$ we have $p_xp_y=0$, so the
subspaces $p_xH$ and $p_yH$ are orthogonal.
For each $x\in X$, choose a unit vector $\xi_x\in p_xH$. Then
$(\xi_x)_{x\in X}$ is an orthonormal family in $H$. Therefore
\(
|X|\le \dim(H).
\)
\end{proof}

\begin{remark}
For effective groupoids and masas on separable Hilbert spaces, that is, for Cartan $C^*$-subalgebras $A\subset B(H)$ with $H$ separable, the argument can be simplified.
Indeed, if $H$ is separable and $G$ is effective, then $C_0(G^{(0)})$
is a Cartan subalgebra of $C_r^*(G,\Sigma)$ by Renault's theorem
\cite{Renault2008}, extended to the non-separable setting by Raad
\cite{Raad22}. In particular, the canonical conditional expectation
onto $C_0(G^{(0)})$ is unique.

If $C_r^*(G,\Sigma)\cong B(H)$, the diagonal $A$ becomes a masa in $B(H)$.
Since $H$ is separable, $A$ is singly generated as a von Neumann algebra by \cite[Proposition~III.1.21]{TakesakiI}.
Thus, by results of Akemann and Sherman \cite{Akemann-Sherman-CE}, the
conditional expectation onto $A$ is normal and $A$ is atomic.
\end{remark}

\subsection{Finite unit space and traces}

Assume that $G$ is étale and $G^{(0)}$ is finite. Then $G$ is discrete.
In this case every twist $\Sigma$ is topologically trivial and is determined
by a $2$-cocycle $\sigma:G^{(2)}\to\T$, which we may assume to be normalized,
that is, $\sigma(r(\gamma),\gamma)=\sigma(\gamma,s(\gamma))=1$ for all $\gamma\in G$.
Applying the cocycle identity to $(\gamma,\gamma^{-1},\gamma)$ yields
\begin{equation}\label{eq:cocycle-symmetry}
\sigma(\gamma,\gamma^{-1})=\sigma(\gamma^{-1},\gamma).
\end{equation}

Define $\tau:C_c(G,\sigma)\to\C$ by
\[
\tau(f)=\frac{1}{|G^{(0)}|}\sum_{u\in G^{(0)}} f(u).
\]
Equivalently, $\tau=\mu\circ E$, where
\[
E:C_r^*(G,\Sigma)\to C_0(G^{(0)})
\]
is the canonical conditional expectation and $\mu$ is the normalized counting measure.

\begin{proposition}\label{prop:trace-finite-units}
If $G$ is discrete and $G^{(0)}$ is finite, then $\tau$ extends to a tracial
state on $C_r^*(G,\Sigma)$.
\end{proposition}

\begin{proof}
It suffices to check the trace identity on $\delta_\alpha,\delta_\beta$.
If $\beta\neq \alpha^{-1}$, then both $\delta_\alpha*\delta_\beta$ and
$\delta_\beta*\delta_\alpha$ vanish on $G^{(0)}$, so $\tau$ gives zero.
If $\beta=\alpha^{-1}$, then
\[
\delta_\alpha * \delta_{\alpha^{-1}}
=\sigma(\alpha,\alpha^{-1})\,\delta_{r(\alpha)},\quad
\delta_{\alpha^{-1}} * \delta_\alpha
=\sigma(\alpha^{-1},\alpha)\,\delta_{s(\alpha)}.
\]
Hence
\[
\tau(\delta_\alpha * \delta_{\alpha^{-1}})
=\frac{1}{|G^{(0)}|}\sigma(\alpha,\alpha^{-1}),\quad
\tau(\delta_{\alpha^{-1}} * \delta_\alpha)
=\frac{1}{|G^{(0)}|}\sigma(\alpha^{-1},\alpha),
\]
which coincide by~\eqref{eq:cocycle-symmetry}. Thus $\tau(f*g)=\tau(g*f)$.
Positivity is clear since $\tau=\mu\circ E$, and continuity yields a tracial state.
\end{proof}

\begin{corollary}\label{cor:finite-units-impossible}
Let $H$ be an infinite-dimensional Hilbert space.
If $B(H)\cong C_r^*(G,\Sigma)$ for some étale groupoid $G$, then $G^{(0)}$ is not finite.
\end{corollary}

\begin{proof}
If $G^{(0)}$ is finite, then $G$ is discrete. By
Proposition~\ref{prop:trace-finite-units}, $C_r^*(G,\Sigma)$ admits a tracial state.
But $B(H)$ admits no tracial state when $H$ is infinite-dimensional.
\end{proof}

\section{Excluding the infinite atomic case}\label{sec:inf-atomic-case}

By Corollary~\ref{cor:A-ellinftyX}, if
\[
B(H)\cong C_r^*(G,\Sigma)
\]
for some infinite-dimensional Hilbert space $H$, then
\[
A:=C_0(G^{(0)})\cong \ell^\infty(X)
\]
for some index set $X$. Corollary~\ref{cor:finite-units-impossible} then implies
that $X$ must be infinite. We now exclude this remaining case.

\begin{theorem}\label{thm:no-linftyX}
Let $H$ be an infinite-dimensional Hilbert space, let $G$ be a locally compact
Hausdorff étale groupoid, and let $\Sigma$ be a twist over $G$. If
\[
C_0(G^{(0)})\cong \ell^\infty(X)
\]
for some infinite set $X$, then
\[
C_r^*(G,\Sigma)\not\cong B(H).
\]
\end{theorem}

\subsection{Atoms and corners}

Assume for contradiction that
\[
B(H)\cong C_r^*(G,\Sigma)
\qquad\text{and}\qquad
A=C_0(G^{(0)})\cong \ell^\infty(X),
\]
with $X$ infinite. Fix an isomorphism
$\Phi:C_r^*(G,\Sigma)\to B(H)$ and henceforth identify elements with their
images under $\Phi$; in particular, regard $A$ as a subalgebra of $B(H)$.
Let $(u_x)_{x\in X}$ be the corresponding family of isolated
points in $G^{(0)}$, and let
\[
p_x:=1_{\{u_x\}}\in A
\qquad (x\in X)
\]
be the minimal projections. Then
\[
H=\bigoplus_{x\in X} H_x,
\qquad
H_x:=p_xH.
\]

We first show that each $H_x$ is finite-dimensional. For $x\in X$, let
\[
G(x):=G_{u_x}^{u_x}=r^{-1}(u_x)\cap s^{-1}(u_x)
\]
be the isotropy group at $u_x$, and let $\Sigma(x):=\Sigma|_{G(x)}$.

\begin{lemma}\label{lem:corner-isotropy}
For each $x\in X$ there is a canonical isomorphism
\[
p_x C_r^*(G,\Sigma)p_x \cong C_r^*(G(x),\Sigma(x)).
\]
\end{lemma}

\begin{proof}
At the level of compactly supported sections,
\[
p_x C_c(G,\Sigma)p_x = C_c(G(x),\Sigma(x)),
\]
since left and right multiplication by $p_x=1_{\{u_x\}}$ cuts the support to
arrows with range and source equal to $u_x$.

It remains to compare the reduced norms. For $f\in C_c(G(x),\Sigma(x))$, viewed
as an element of $C_c(G,\Sigma)$, the reduced norm in $C_r^*(G,\Sigma)$ is
\[
\|f\|_r=\sup_{u\in G^{(0)}} \|\lambda_u(f)\|,
\]
where $\lambda_u$ denotes the regular representation at $u$. If $u\neq u_x$,
then $\lambda_u(f)=0$, because $f$ is supported on arrows with source $u_x$.
For $u=u_x$, the representation $\lambda_{u_x}$ restricts exactly to the
regular representation of the discrete twisted group $(G(x),\Sigma(x))$.
Therefore
\[
\|f\|_{C_r^*(G,\Sigma)}=\|\lambda_{u_x}(f)\|
=\|f\|_{C_r^*(G(x),\Sigma(x))}.
\]
So the inclusion
\[
C_c(G(x),\Sigma(x))\hookrightarrow p_xC_c(G,\Sigma)p_x
\]
is isometric for the reduced norms, and completion yields the result.
\end{proof}

\begin{lemma}\label{lem:corner-trace}
For each $x\in X$, the corner $p_x C_r^*(G,\Sigma)p_x$ admits a faithful
tracial state.
\end{lemma}
\begin{proof}
By Lemma~\ref{lem:corner-isotropy}, it suffices to consider
$C_r^*(G(x),\Sigma(x))$. Since $G(x)$ is a discrete group, the canonical trace
is given on $C_c(G(x),\Sigma(x))$ by
\[
\tau_x(f)=f(u_x).
\]
This extends to a faithful tracial
state on $C_r^*(G(x),\Sigma(x))$.
\end{proof}

On the other hand,
\[
p_x C_r^*(G,\Sigma)p_x \cong p_xB(H)p_x \cong B(H_x).
\]

\begin{corollary}\label{cor:Hx-finite}
For every $x\in X$, the Hilbert space $H_x$ is finite-dimensional.
\end{corollary}

\begin{proof}
By Lemma~\ref{lem:corner-trace}, the algebra $B(H_x)$ admits a faithful tracial
state. This is impossible if $H_x$ were infinite-dimensional.
\end{proof}

\subsection{Compact support implies block sparsity}

Let $f\in C_c(G,\Sigma)$ and set $K:=\supp(f)$. Since $G$ is étale, every point
of $K$ admits an open bisection neighbourhood, and by compactness there exist
open bisections $U_1,\dots,U_m$ such that $K\subseteq U_1\cup\cdots\cup U_m$.

\begin{lemma}\label{lem:fiber-count}
For each $x\in X$,
\[
|K\cap r^{-1}(u_x)|\le m,
\qquad
|K\cap s^{-1}(u_x)|\le m.
\]
\end{lemma}

\begin{proof}
Since each $U_j$ is a bisection, both $r|_{U_j}$ and $s|_{U_j}$ are injective.
Hence each $U_j\cap r^{-1}(u_x)$ and each $U_j\cap s^{-1}(u_x)$ contains at most one
point. Summing over $j=1,\dots,m$ gives the result.
\end{proof}

Let $T_f\in B(H)$ be the image of $f$ under the fixed isomorphism. Since the
mutually orthogonal projections $(p_x)_{x\in X}$ have strong sum $1$, every
operator on $H=\bigoplus_{x\in X}H_x$ is determined by its block matrix. Thus write
\[
T_f=(T_{yx})_{x,y\in X},
\qquad
T_{yx}:=p_yT_fp_x\in B(H_x,H_y).
\]

\begin{lemma}\label{lem:block-support}
For all $x,y\in X$,
\[
\supp(p_y f p_x)\subseteq K\cap r^{-1}(u_y)\cap s^{-1}(u_x).
\]
In particular:
\begin{enumerate}[label=\rm(\alph*)]
\item for each fixed $x\in X$, there are at most $m$ elements $y\in X$ such that
$T_{yx}\neq 0$;
\item for each fixed $y\in X$, there are at most $m$ elements $x\in X$ such that
$T_{yx}\neq 0$.
\end{enumerate}
\end{lemma}

\begin{proof}
Left multiplication by $p_y$ forces the range to be $u_y$, and right
multiplication by $p_x$ forces the source to be $u_x$, so
\[
\supp(p_y f p_x)\subseteq K\cap r^{-1}(u_y)\cap s^{-1}(u_x).
\]
If $T_{yx}\neq 0$, then $p_y f p_x\neq 0$, hence this intersection is nonempty.

For fixed $x$, distinct $y$ give distinct elements of $K\cap s^{-1}(u_x)$,
so Lemma~\ref{lem:fiber-count} yields at most $m$ possibilities. The row estimate
is analogous.
\end{proof}

\subsection{The sparse classes $S_k$}

For $k\in\N$, let $S_k$ be the set of all block operators
$T=(T_{yx})_{x,y\in X}\in B(H)$ such that every row and every column has at
most $k$ nonzero blocks, i.e.
\[
\sup_{x\in X} |\{y\in X:T_{yx}\neq 0\}|\le k,
\qquad
\sup_{y\in X} |\{x\in X:T_{yx}\neq 0\}|\le k.
\]

By Lemma~\ref{lem:block-support}, every $T_f$ with $f\in C_c(G,\Sigma)$
belongs to some $S_k$.

\begin{lemma}\label{lem:Sk-stability}
If $T\in S_k$ and $R\in S_\ell$, then $T+R\in S_{k+\ell}$,
$TR\in S_{k\ell}$, and $T^*\in S_k$.
\end{lemma}

\begin{proof}
The statement for sums is immediate, and taking adjoints exchanges rows and
columns, so $T^*\in S_k$. For products, fix $x\in X$. There are at
most $\ell$ elements $z$ with $R_{zx}\neq 0$, and for each such $z$ there are at
most $k$ elements $y$ with $T_{yz}\neq 0$. Since
\[
(TR)_{yx}=\sum_{z\in X} T_{yz}R_{zx},
\]
and the sum is finite, there are at most $k\ell$ elements $y$ with $(TR)_{yx}\neq 0$.
The row estimate is analogous.
\end{proof}

\begin{corollary}\label{cor:dense-subalg-sparse}
The $*$-subalgebra generated by $C_c(G,\Sigma)$ is contained in $\bigcup_{k\ge1}S_k$.
Consequently,
\[
C_r^*(G,\Sigma)\subseteq \overline{\bigcup_{k\ge1}S_k}^{\|\cdot\|}.
\]
\end{corollary}

\begin{proof}
Each compactly supported section belongs to some $S_k$, and
Lemma~\ref{lem:Sk-stability} shows that finite sums, products, and adjoints remain
in $\bigcup_k S_k$. Taking the norm closure gives the result.
\end{proof}

\subsection{A spreading operator}

Since $X$ is infinite, we may choose pairwise disjoint finite subsets
$X_1,X_2,\dots\subseteq X$ with $|X_r|=r$ for all $r\ge1$, and points
$j_r\in X_r$. For each $x\in X$, choose a unit vector $\eta_x\in H_x$.

Define
\[
\xi_r:=\frac1{\sqrt r}\sum_{x\in X_r}\eta_x.
\]
Since the $H_x$ are mutually orthogonal and the sets $X_r$ are disjoint, the
vectors $\xi_r$ form an orthonormal family.

Let $V\in B(H)$ be defined by
\[
V\eta_{j_r}=\xi_r \qquad (r\ge1),
\]
and $V=0$ on the orthogonal complement of
$\span\{\eta_{j_r}:r\ge1\}$. Then $V$ is a partial isometry.

\begin{lemma}\label{lem:distance-estimate}
If $T\in S_k$, then $\|T-V\|\ge1$. In particular,
\[
V\notin \overline{\bigcup_{k\ge1}S_k}^{\|\cdot\|}.
\]
\end{lemma}

\begin{proof}
Fix $r>k$. Since $T\in S_k$, the $j_r$-th column of $T$ has at most $k$ nonzero
blocks. Hence there exists a set $F_r\subseteq X$ with $|F_r|\le k$ such that
\[
T\eta_{j_r}\in \bigoplus_{x\in F_r} H_x.
\]
On the other hand,
\[
V\eta_{j_r}=\xi_r=\frac1{\sqrt r}\sum_{x\in X_r}\eta_x.
\]

The orthogonal projection of $\xi_r$ onto $\bigoplus_{x\in F_r}H_x$ is
\[
\frac1{\sqrt r}\sum_{x\in F_r\cap X_r}\eta_x,
\]
whose squared norm is
\[
\frac{|F_r\cap X_r|}{r}\le \frac{k}{r}.
\]
Therefore
\[
\dist\!\left(\xi_r,\bigoplus_{x\in F_r}H_x\right)\ge \sqrt{1-\frac{k}{r}}.
\]
Since $T\eta_{j_r}\in \bigoplus_{x\in F_r}H_x$, it follows that
\[
\|T\eta_{j_r}-\xi_r\|
\ge \dist\!\left(\xi_r,\bigoplus_{x\in F_r}H_x\right)
\ge \sqrt{1-\frac{k}{r}}.
\]
Hence
\[
\|T-V\|\ge \|T\eta_{j_r}-V\eta_{j_r}\|
\ge \sqrt{1-\frac{k}{r}}.
\]
Letting $r\to\infty$ gives $\|T-V\|\ge1$.
\end{proof}

\subsection{Conclusion of the infinite case}

\begin{proof}[Proof of Theorem~\ref{thm:no-linftyX}]
By Corollary~\ref{cor:dense-subalg-sparse},
\[
C_r^*(G,\Sigma)\subseteq \overline{\bigcup_{k\ge1}S_k}^{\|\cdot\|}.
\]
But $V\notin \overline{\bigcup_{k\ge1}S_k}^{\|\cdot\|}$ by
Lemma~\ref{lem:distance-estimate}. Hence $C_r^*(G,\Sigma)\neq B(H)$.
\end{proof}

\section{Proof of the main theorem}

\begin{proof}[Proof of Theorem~\ref{thm:main}]
Assume, towards a contradiction, that
\[
B(H)\cong C_r^*(G,\Sigma)
\]
for some infinite-dimensional Hilbert space $H$, some locally compact Hausdorff
étale groupoid $G$, and some twist $\Sigma$. Let
\[
A=C_0(G^{(0)})\subseteq C_r^*(G,\Sigma)\cong B(H).
\]
By Corollary~\ref{cor:A-ellinftyX}, we have
\[
A\cong \ell^\infty(X)
\]
for some index set $X$.

If $X$ is finite, then $G^{(0)}$ is finite, contradicting
Corollary~\ref{cor:finite-units-impossible}. If $X$ is infinite, then
Theorem~\ref{thm:no-linftyX} implies that $C_r^*(G,\Sigma)\neq B(H)$,
again a contradiction. Thus no such $(G,\Sigma)$ exists.
\end{proof}

\begin{remark}
The argument does not require $G$ to be effective or topologically principal.
In particular, it excludes all étale twisted groupoid models for $B(H)$.
\end{remark}

\section{Open questions}

The result raises several natural questions.

\begin{enumerate}[label=\rm(\arabic*)]
\item Does the analogue of Theorem~\ref{thm:main} hold for the full twisted
$C^*$-algebra $C^*(G,\Sigma)$?

\item Does the conclusion remain valid if $G$ is not assumed to be étale?

\item What happens in the non-Hausdorff setting?

\item More generally, which von Neumann algebras can be realized as reduced
twisted groupoid $C^*$-algebras? For instance, can one obtain examples among
type~II algebras such as group von Neumann algebras?

\item Does there exist a separable $C^*$-algebra which is not isomorphic to
$C_r^*(G,\Sigma)$ for any locally compact (Hausdorff, étale) groupoid $G$ and
twist $\Sigma$?
\end{enumerate}

\bibliographystyle{alphaurl}

\end{document}